\documentclass[reqno]{amsart}
\usepackage{amssymb,hyperref}
\usepackage[dvips]{graphicx}

\makeatletter

\@addtoreset{equation}{section}

\theoremstyle{plain}
\newtheorem{thm}{Theorem}[section]

\newtheorem{ex}{Example}[section]
\newtheorem{conj}{Conjecture}[section]
\theoremstyle{remark}
\newtheorem{remark}{Remark}[section]

\renewcommand{\Re}{\mathrm{Re}}
\renewcommand{\Im}{\mathrm{Im}}

\title
[A sufficient condition for $p$-valently harmonic functions]
{A sufficient condition for\\
\lowercase{$p$}-valently harmonic functions}

\author{Toshio Hayami}
\address{Toshio Hayami \newline
Department of Mathematics, \newline
Kinki University \newline
Higashi-Osaka, Osaka 577-8502, \newline
Japan}
\email{ha\_ya\_to112@hotmail.com}

\date{}

\begin{document}

\begin{abstract}
For normalized harmonic functions $f(z)=h(z)+\overline{g(z)}$ in the open unit disk $\mathbb{U}$, a sufficient condition on $h(z)$ for $f(z)$ to be $p$-valent in $\mathbb{U}$ is discussed. Moreover, some interesting examples and images of $f(z)$ satisfying the obtained condition are enumerated.
\end{abstract}

\maketitle

\section{Introduction and Definitions}

\

For a fixed $p$ $(p=1,2,3,\cdots)$, a meromorphic function $f(z)$ in a domain $\mathbb{D}$ is said to be $p$-valent (or multivalent of order $p$) in $\mathbb{D}$ if for each $w_0$ (infinity included) the equation $f(z)=w_0$ has at most $p$ roots in $\mathbb{D}$ where the roots are counted in accordance with their multiplicity and if there is some $w_1$ such that the equation $f(z)=w_1$ has exactly $p$ roots in $\mathbb{D}$. In particular, $f(z)$ is said to be univalent (one-to-one) in $\mathbb{D}$ when $p=1$.

A complex-valued harmonic function $f(z)$ in $\mathbb{D}$ is given by $f(z)=h(z)+\overline{g(z)}$ where $h(z)$ and $g(z)$ are analytic in $\mathbb{D}$. We call $h(z)$ and $g(z)$ the analytic part and co-analytic part of $f(z)$, respectively. A necessary and sufficient condition for $f(z)$ to be locally univalent and sense-preserving in $\mathbb{D}$ is $|h'(z)|>|g'(z)|$ for all $z\in \mathbb{D}$ (see \cite{CS} or \cite{L}). The theory and applications of harmonic functions are stated in a book due to Duren \cite{D}. Let $\mathcal{H}(p)$ denote the class of functions $f(z)$ of the from
\begin{equation} \label{deff(z)}
f(z)=h(z)+\overline{g(z)}=z^p+\sum\limits_{n=p+1}^{\infty}a_n z^n+\overline{\sum\limits_{n=p}^{\infty}b_n z^n}
\end{equation}
which are harmonic in the open unit disk $\mathbb{U}=\left\{z\in \mathbb{C}:|z|<1\right\}$. We next denote by $\mathcal{S}_{\mathcal{H}}(p)$ the class of functions $f(z)\in \mathcal{H}(p)$ which are $p$-valent and sense-preserving in $\mathbb{U}$. Then, we say that $f(z)\in \mathcal{S}_{\mathcal{H}}(p)$ is a $p$-valently harmonic function in $\mathbb{U}$.

In the present paper, we discuss a sufficient condition about $h(z)$ for $f(z)\in \mathcal{H}(p)$ given by (\ref{deff(z)}), satisfying
\begin{equation}
g'(z)=z^{m-1}h'(z)
\end{equation}
for some $m$ $(m=2,3,4,\cdots)$, to be in the class $\mathcal{S}_{\mathcal{H}}(p)$. In this case, we can write
\begin{equation}
f(z)=h(z)+\overline{z^{m-1}h(z)-(m-1)\int_{0}^{z}\zeta^{m-2}h(\zeta)d\zeta}
\end{equation}
which means that $f(z)$ is well defined if $h(z)$ is given.

\

\section{Main Result}

\

Our result is contained in

\begin{thm} \label{thm1}
Let $h(z)=z^p+\sum\limits_{n=p+1}^{\infty}a_n z^n$ be analytic in the closed unit disk $\overline{\mathbb{U}}=\left\{z\in \mathbb{C}:|z|\leqq 1\right\}$ with $H(z)=h'(z)/z^{p-1}\neq 0\ \ (z\in \overline{\mathbb{U}})$ and let
\begin{equation} \label{defF(t)}
F(t)=(2p+m-1)t+2\arg\left(H(e^{it})\right)\qquad (-\pi\leqq t<\pi)
\end{equation}
for some $m$ $(m=2,3,4,\cdots)$. If for each $k\in K=\left\{0,\pm 1,\pm 2,\cdots,\pm \left[\frac{2p+m+1}{2}\right]\right\}$ where $[\ ]$ is the Gauss symbol, the equation
\begin{equation} \label{eqF(t)}
F(t)=2k\pi
\end{equation}
has at most a single root in $\left[\left.-\pi,\pi\right)\right.$ and for all $k\in K$ there exist exactly $2p+m-1$ such roots in $\left[\left.-\pi,\pi\right)\right.$, then the harmonic function $f(z)=h(z)+\overline{g(z)}$ with $g'(z)=z^{m-1}h'(z)$ belongs to the class $\mathcal{S}_{\mathcal{H}}(p)$ and maps $\mathbb{U}$ onto a domain surrounded by $2p+m-1$ concave curves with $2p+m-1$ cusps.
\end{thm}

\begin{proof}
We first consider the function $\varphi(t)$ defined as
\begin{equation} \label{phi(t)}
\varphi(t)=f(e^{it})=h(e^{it})+\overline{g(e^{it})}\qquad (-\pi\leqq t<\pi).
\end{equation}
Supposing that
\begin{equation}
\varphi'(t)=i\left(zh'(z)-\overline{z}\overline{g'(z)}\right)=iz\left(h'(z)-\overline{z}^{m+1}\overline{h'(z)}\right)=0\qquad (z=e^{it}),
\end{equation}
we need the following equation
\begin{equation}
z^{m+1}\dfrac{h'(z)}{\overline{h'(z)}}=1\qquad \left(\overline{z}=e^{-it}=\dfrac{1}{z}\right).
\end{equation}
This means that
\begin{eqnarray}
\arg\left(e^{i(m+1)t}\dfrac{h'(e^{it})}{\overline{h'(e^{it})}}\right) &=& (m+1)t+2\arg\left(e^{i(p-1)t}H(e^{it})\right) \nonumber \\
 &=& (2p+m-1)t+2\arg\left(H(e^{it})\right)=2k\pi
\end{eqnarray}
for some $k\in K$ which gives us the equation (\ref{eqF(t)}). In consideration of the assumption of the theorem, there are $2p+m-1$ distinct roots on the unit circle $\partial \mathbb{U}=\left\{z\in \mathbb{C}:|z|=1\right\}$ and they divide $\partial \mathbb{U}$ onto $2p+m-1$ arcs. Moreover, since $g''(z)=(m-1)z^{m-2}h'(z)+z^{m-1}h''(z)$, we obtain that
\begin{eqnarray}
\varphi''(t) &=& -\left(zh'(z)+z^2 h''(z)+\overline{z}\overline{g'(z)}+\overline{z}^2\overline{g''(z)}\right) \nonumber \\
 &=& -\left(zh'(z)+z^2 h''(z)+m\overline{z}^m \overline{h'(z)}+\overline{z}^{m+1}\overline{h''(z)}\right)
\end{eqnarray}
and
\begin{eqnarray}
\varphi''(t)\overline{\varphi'(t)} &=& -\left(zh'(z)+z^2 h''(z)+m\overline{z}^m \overline{h'(z)}+\overline{z}^{m+1}\overline{h''(z)}\right)(-i)\left(\overline{z}\overline{h'(z)}-z^m h'(z)\right) \nonumber \\
 & & \nonumber \\
 &=& i\left(-(m-1)|h'(z)|^2+z\overline{h'(z)}h''(z)-\overline{z\overline{h'(z)}h''(z)}-z^{m+1}h'(z)^2\right. \nonumber \\
 & & \left.+\overline{z^{m+1}h'(z)^2}+(m-1)\overline{z}^{m+1}\overline{h'(z)}^2-z^{m+2}h'(z)h''(z)+\overline{z^{m+2}h'(z)h''(z)}\right) \nonumber \\
 & & \nonumber \\
 &=& i\left\{(m-1)\left(\overline{z}^{m+1}\overline{h'(z)}^2-|h'(z)|^2\right)\right. \nonumber \\
 & & \left.+2i\Im\left(z\overline{h'(z)}h''(z)-z^{m+1}h'(z)^2-z^{m+2}h'(z)h''(z)\right)\right\}.
\end{eqnarray}
This leads us that
\begin{equation}
\Im\left(\varphi''(t)\overline{\varphi'(t)}\right)=(m-1)|h'(z)|^2\Re\left\{\overline{z}^{m+1}\left(\dfrac{\overline{h'(z)}}{|h'(z)|}\right)^2-1\right\}\leqq 0
\end{equation}
for all $z=e^{it}$. Therefore, it follows that
\begin{equation}
\Im\left(\dfrac{\varphi''(t)}{\varphi'(t)}\right)=\dfrac{1}{|\varphi'(t)|^2}\Im\left(\varphi''(t)\overline{\varphi'(t)}\right)\leqq 0
\end{equation}
which shows that $\varphi(t)$ maps $[-\pi,\pi)$ onto a union of $2p+m-1$ concave curves. By the help of a simple geometrical observation, we know that the image of $\partial \mathbb{U}$ as a union of $2p+m-1$ concave arcs is a simple curve. Thus, $f(z)$ is $p$-valent in $\mathbb{U}$ and the image $f(\mathbb{U})$ is a domain surrounded by $2p+m-1$ concave curves with $2p+m-1$ cusps.
\end{proof}

\begin{remark}
If we take $p=1$ in Theorem \ref{thm1}, then we readily arrive at the univalence criterion for harmonic functions due to Hayami and Owa \cite[Theorem 2.1]{HO} (see also \cite{M1}).
\end{remark}

\

\section{Some Illustrative Examples and Image Domains}

\

In this section, we discuss functions $f(z)=h(z)+\overline{g(z)}$ satisfying the conditions of Theorem \ref{thm1} and their image domains.

\begin{ex}
Let $h(z)=z^p$. Then we easily see that the equation $(\ref{eqF(t)})$ becomes
\begin{equation}
(2p+m-1)t=2k\pi\qquad \left(k=0,\pm 1,\pm 2,\cdots,\pm \left[\dfrac{2p+m+1}{2}\right]\right)
\end{equation}
which satisfies the conditions of Theorem \ref{thm1}. Hence, the function
\begin{equation}
f(z)=h(z)+\overline{g(z)}=z^p+\overline{\dfrac{p}{p+m-1}z^{p+m-1}}\qquad (g'(z)=z^{m-1}h'(z))
\end{equation}
belongs to the class $\mathcal{S}_{\mathcal{H}}(p)$ and it maps $\mathbb{U}$ onto a domain surrounded by $2p+m-1$ concave curves with $2p+m-1$ cusps. Taking $p=2$ and $m=4$ for example, we know that the function
\begin{equation}
f(z)=z^2+\dfrac{2}{5}\overline{z}^5
\end{equation}
is a $2$-valently harmonic function in $\mathbb{U}$ and it maps $\mathbb{U}$ onto the domain surrounded by $7$ concave curves with $7$ cusps as shown in Figure 1.
\end{ex}

\begin{figure*}[h]
\begin{center}
\includegraphics[width=7cm,clip]{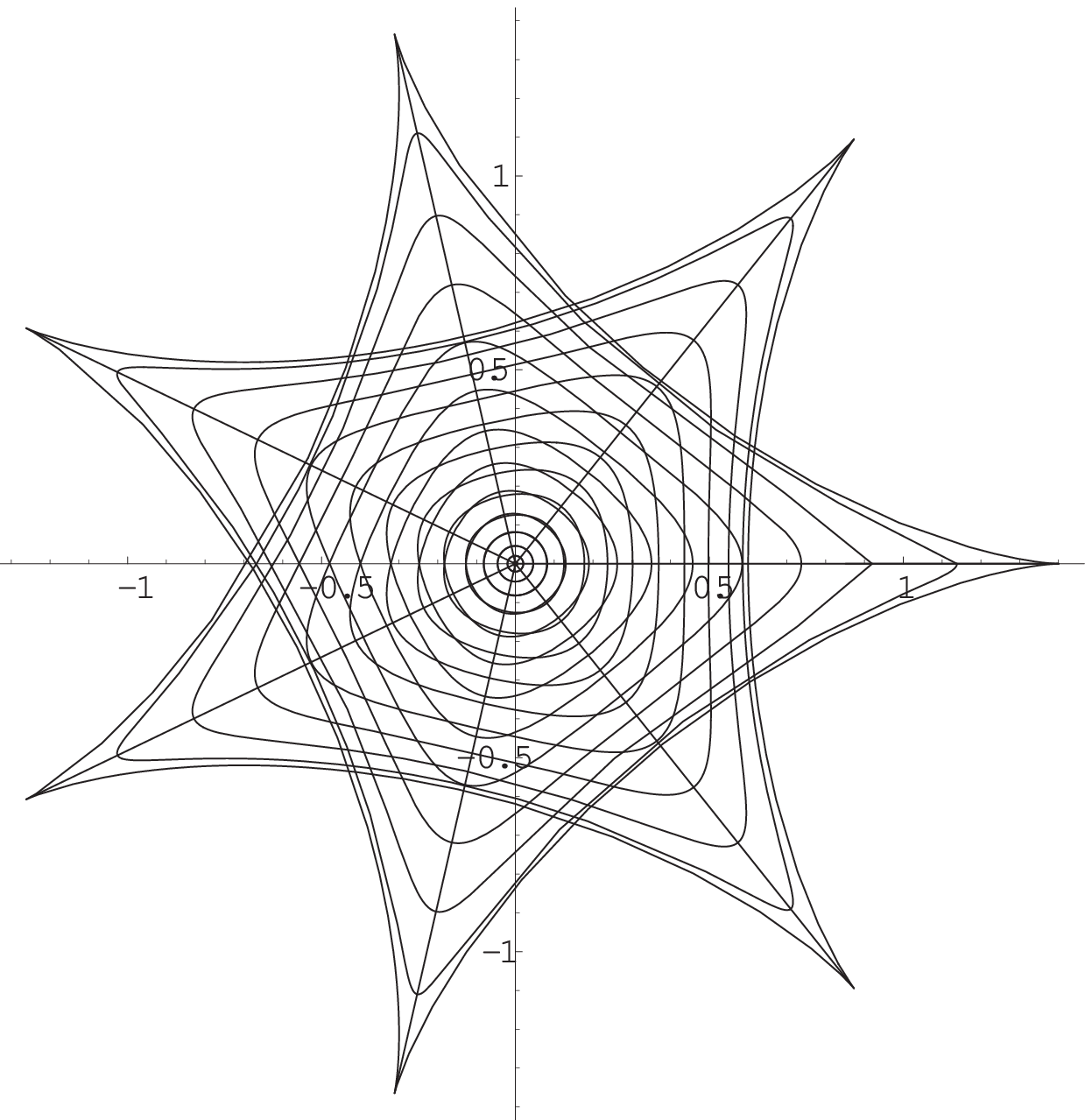}
\end{center}
\caption{The image of $f(z)=z^2+\dfrac{2}{5}\overline{z}^5$.}
\end{figure*}

\begin{remark} \label{rem}
Since we can rewrite $F(t)$ given by $(\ref{defF(t)})$ as follows:
\begin{equation}
F(t)=(2p+m-1)t+2\Im\left(\log H(e^{it})\right),
\end{equation}
we have that
\begin{eqnarray}
F'(t) &=& 2p+m-1+2\Re\left(\dfrac{e^{it}h''(e^{it})}{h'(e^{it})}-(p-1)\right) \nonumber \\
 & & \nonumber \\
 &=& m+1+2\Re\left(\dfrac{e^{it}h''(e^{it})}{h'(e^{it})}\right)
\end{eqnarray}
which implies that $F(t)$ is increasing if
\begin{equation}
\Re\left(1+\dfrac{zh''(z)}{h'(z)}\right)>-\dfrac{m-1}{2}\qquad (z\in \mathbb{U}).
\end{equation}
\end{remark}

\

By the above remark, we derive the following exapmle.
\begin{ex}
Let $h(z)=z^p+\dfrac{c}{p+1}z^{p+1}$ $\left(|c|\leqq p-\dfrac{2p}{2p+m+1}\right)$. Then the equation $(\ref{eqF(t)})$ becomes
\begin{equation}
F(t)=(2p+m-1)t+2\arg\left(p+ce^{it}\right).
\end{equation}
Noting
\begin{eqnarray}
\Re\left(1+\dfrac{zh''(z)}{h'(z)}\right) &=& p+1-\Re\left(\dfrac{p}{p+cz}\right) \nonumber \\
 & & \nonumber \\
 &>& p+1-\dfrac{p}{p-|c|}\ \geqq -\dfrac{m-1}{2}\qquad (z\in \mathbb{U}),
\end{eqnarray}
\begin{equation}
F(-\pi)=-(2p+m-1)\pi-2\arctan\left(\dfrac{|c|\sin \theta}{p-|c|\cos \theta}\right)
\end{equation}
and
\begin{equation}
F(\pi)=(2p+m-1)\pi-2\arctan\left(\dfrac{|c|\sin \theta}{p-|c|\cos \theta}\right)
\end{equation}
where $0\leqq \theta=\arg(c)<2\pi$, we see that $F(t)$ satisfies the conditions of Theorem \ref{thm1}. Hence, the function
\begin{equation}
f(z)=h(z)+\overline{g(z)}=z^p+\dfrac{c}{p+1}z^{p+1}+\overline{\dfrac{p}{p+m-1}z^{p+m-1}+\dfrac{c}{p+m}z^{p+m}}
\end{equation}
belongs to the class $\mathcal{S}_{\mathcal{H}}(p)$ and it maps $\mathbb{U}$ onto a domain surrounded by $2p+m-1$ concave curves with $2p+m-1$ cusps. Putting $p=3$, $m=2$ and $c=i$ $\left(|c|=1<\dfrac{7}{3}\right)$, we know that the function
\begin{equation}
f(z)=z^3+\dfrac{i}{4}z^4+\dfrac{3}{4}\overline{z}^4-\dfrac{i}{5}\overline{z}^5
\end{equation}
is a $3$-valently harmonic function in $\mathbb{U}$ and it maps $\mathbb{U}$ onto the domain surrounded by $7$ concave curves with $7$ cusps as shown in Figure 2. We check only the boundary of $f(\mathbb{U})$, for the sake of simplicity.
\end{ex}

\begin{figure*}[h]
\begin{center}
\includegraphics[width=7cm,clip]{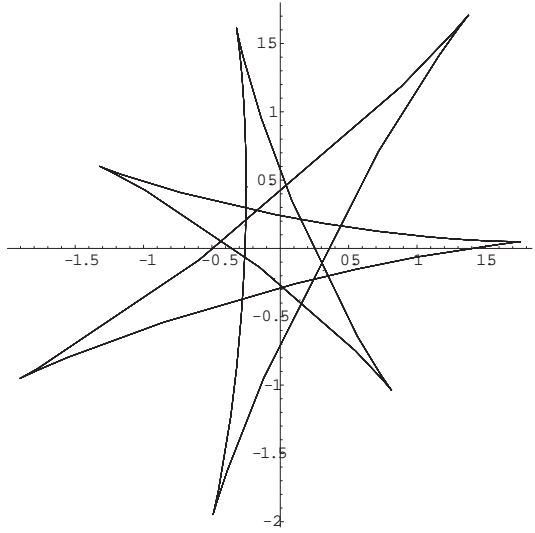}
\end{center}
\caption{The image of $f(z)=z^3+\dfrac{i}{4}z^4+\dfrac{3}{4}\overline{z}^4-\dfrac{i}{5}\overline{z}^5$.}
\end{figure*}

\

\section{Appendix}

\

The next result was conjectured by Mocanu \cite{M2} and proved by Bshouty and Lyzzaik \cite{BL}.

\begin{thm}
If $h(z)$ and $g(z)$ are analytic in $\mathbb{U}$, with $h'(0)\neq 0$, which satisfy
\begin{equation}
g'(z)=zh'(z)
\end{equation}
and
\begin{equation}
\Re\left(1+\dfrac{zh''(z)}{h'(z)}\right)>-\dfrac{1}{2}
\end{equation}
for all $z\in \mathbb{U}$, then the harmonic function $f(z)=h(z)+\overline{g(z)}$ is univalent close-to-convex in $\mathbb{U}$.
\end{thm}

We immediately notice that the above theorem is closely related to Theorem \ref{thm1} and Remark \ref{rem} with $p=1$ and $m=2$. This motivates us to state

\begin{conj}
If the function $f(z)$ given by $(\ref{deff(z)})$ is harmonic in $\mathbb{U}$ which satisfies
\begin{equation}
g'(z)=z^{m-1}h'(z)
\end{equation}
and
\begin{equation}
\Re\left(1+\dfrac{zh''(z)}{h'(z)}\right)>-\dfrac{m-1}{2}\qquad (z\in \mathbb{U})
\end{equation}
for some $m$ $(m=2,3,4,\cdots)$, then $f(z)$ is $p$-valent in $\mathbb{U}$.
\end{conj}

\

Finally, in view of the process of proving Theorem \ref{thm1}, we obtain the following interesting example.
\begin{ex}
If we consider special functions $h(z)$ and $g(z)$ given by
\begin{equation}
h'(z)=\dfrac{pz^{p-1}}{1+z^{2p+m-1}}\quad and\quad g'(z)=\dfrac{pz^{p+m-2}}{1+z^{2p+m-1}}\qquad \left(g'(z)=z^{m-1}h'(z)\right),
\end{equation}
then the function $\varphi(t)$ given by $(\ref{phi(t)})$ satisfies
\begin{eqnarray}
\varphi'(t) &=& iz\left(h'(z)-\overline{z^{m+1}h'(z)}\right) \nonumber \\
 & & \nonumber \\
 &=& iz\left(\dfrac{pz^{p-1}}{1+z^{2p+m-1}}-\dfrac{p\overline{z}^{p+m}}{1+\overline{z}^{2p+m-1}}\right) \nonumber \\
 & & \nonumber \\
 &=& iz\left(\dfrac{pz^{p-1}}{1+z^{2p+m-1}}-\dfrac{pz^{p-1}}{1+z^{2p+m-1}}\right)=0\qquad (z=e^{it})
\end{eqnarray}
for any $z^{2p+m-1}=e^{i(2p+m-1)t}\neq -1$. Thus, the function
\begin{equation} \label{specialf(z)}
f(z)=h(z)+\overline{g(z)}=\int_{0}^{z}\dfrac{p\zeta^{p-1}}{1+\zeta^{2p+m-1}}d\zeta+\overline{\int_{0}^{z}\dfrac{p\zeta^{p+m-2}}{1+\zeta^{2p+m-1}}d\zeta}
\end{equation}
is a member of the class $\mathcal{S}_{\mathcal{H}}(p)$ and it maps $\mathbb{U}$ onto a domain surrounded by $2p+m-1$ straight lines with $2p+m-1$ cusps. Indeed, setting $p=2$ and $m=2$, we know that
\begin{equation}
f(z)=\int_{0}^{z}\dfrac{2\zeta}{1+\zeta^5}d\zeta+\overline{\int_{0}^{z}\dfrac{2\zeta^{2}}{1+\zeta^5}d\zeta}
\end{equation}
is a $2$-valently harmonic function and it maps $\mathbb{U}$ onto a star as shown in Figure 3. Furthermore, if we take $p=1$ in $(\ref{specialf(z)})$, then we see that the function
\begin{equation}
f_{m+1}(z)=h(z)+\overline{g(z)}=\int_{0}^{z}\dfrac{1}{1+\zeta^{m+1}}d\zeta+\overline{\int_{0}^{z}\dfrac{\zeta^{m-1}}{1+\zeta^{m+1}}d\zeta}
\end{equation}
is univalent in $\mathbb{U}$ and it maps $\mathbb{U}$ onto a $(m+1)$-sided polygon. For example, the function
\begin{equation}
f_{8}(z)=\int_{0}^{z}\dfrac{1}{1+\zeta^{8}}d\zeta+\overline{\int_{0}^{z}\dfrac{\zeta^{6}}{1+\zeta^{8}}d\zeta}\qquad (m=7)
\end{equation}
maps $\mathbb{U}$ onto an octagon (schlicht domain) as shown in Figure 4.
\end{ex}

\begin{figure*}[h]
\begin{center}
\includegraphics[width=7cm,clip]{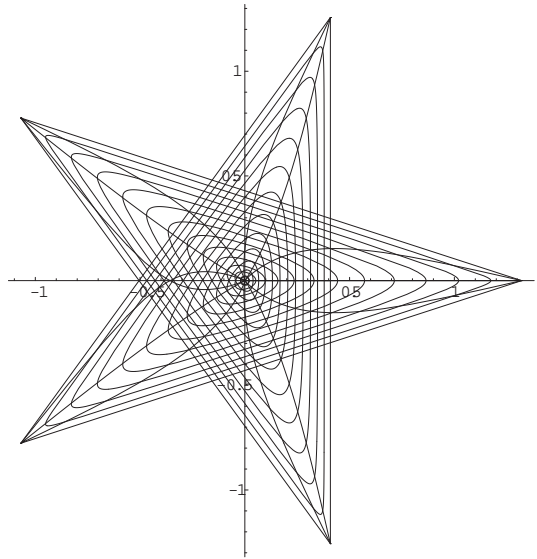}
\end{center}
\caption{The image of ${\displaystyle f(z)=\int_{0}^{z}\dfrac{2\zeta}{1+\zeta^5}d\zeta+\overline{\int_{0}^{z}\dfrac{2\zeta^{2}}{1+\zeta^5}d\zeta}}$.}
\end{figure*}

\

\begin{figure*}[h]
\begin{center}
\includegraphics[width=7cm,clip]{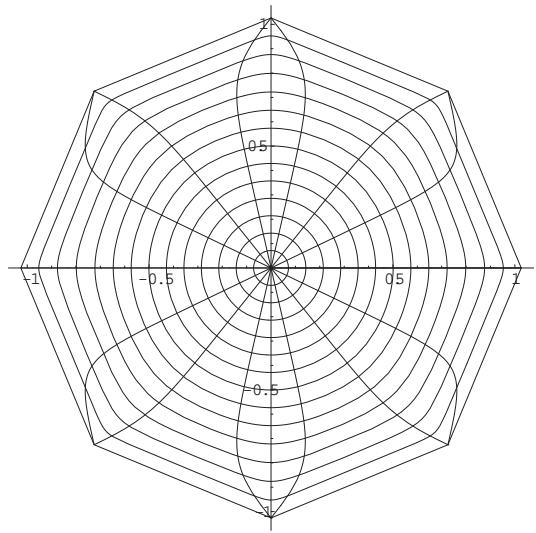}
\end{center}
\caption{The image of ${\displaystyle f_{8}(z)=\int_{0}^{z}\dfrac{1}{1+\zeta^{8}}d\zeta+\overline{\int_{0}^{z}\dfrac{\zeta^{6}}{1+\zeta^{8}}d\zeta}}$.}
\end{figure*}

\


\begin{thebibliography}{}

\bibitem{BL}
D. Bshouty and A. Lyzzaik,
{\it Close-to-convexity criteria for planar harmonic mappings},
Complex Anal. Oper. Theory {\bf 5}(2011), 767--774.

\bibitem{CS}
J. Clunie and T. Sheil-Small,
{\it Harmonic univalent functions},
Ann. Acad. Sci. Fenn. Ser. A I Math. {\bf 9}(1984), 3--25.

\bibitem{D}
P. L. Duren,
{\it Harmonic Mappings in the Plane},
Cambridge University Press, Cambridge, 2004.

\bibitem{G}
A. W. Goodman,
{\it Univalent Functions, Vol. I and Vol. I\hspace{-0.15em}I}, Mariner Publishing Company, Tampa, Florida (1983).

\bibitem{HO}
T. Hayami and S. Owa,
{\it Hypocycloid of $n+1$ cusps harmonic function},
Bull. Math. Anal. Appl. {\bf 3}(2011), 239--246.

\bibitem{L}
H. Lewy,
{\it On the non-vanishing of the Jacobian in certain one-to-one mappings},
Bull. Amer. Math. Soc. {\bf 42}(1936), 689--692.

\bibitem{M1}
P. T. Mocanu,
{\it Three-cornered hat harmonic functions},
Complex Var. Elliptic Equ. {\bf 54}(2009), 1079--1084.

\bibitem{M2}
P. T. Mocanu,
{\it Injectively conditions in the complex plane},
Complex Anal. Oper. Theory {\bf 5}(2011), 759--766.

\end{thebibliography}
\end{document}